\documentclass[11pt]{article}
\usepackage[english]{babel}
\usepackage{amssymb,amsmath,enumerate,a4wide,amsthm,soul}
\usepackage{colortab}
\usepackage{color}
\usepackage{epsfig}
\usepackage{fullpage}
\usepackage{algorithmic}
\usepackage[capitalize,nameinlink]{cleveref}
\usepackage{algorithm}

\newcommand{\beq}{\begin{eqnarray}}
\newcommand{\eeq}{\end{eqnarray}}
\newcommand{\bb}{\mathbb}
\newcommand{\R}{\bb R}

\newcommand{\Z}{\bb Z}

\newcommand{\conv}{\operatorname{conv}}

\newcommand{\aff}{\operatorname{aff}}

\renewcommand{\epsilon}{\varepsilon}

\newtheorem{prop}{Proposition}
\newtheorem{definition}[prop]{Definition}
\newtheorem{remark}[prop]{Remark}
\newtheorem{lemma}[prop]{Lemma}
\newtheorem{cor}[prop]{Corollary}
\newtheorem{theorem}[prop]{Theorem}

\newtheorem{observation}[prop]{Observation}

\newtheorem{assumption}[prop]{Assumption}

\newtheorem*{claim-no-number}{\sc Claim}

\Crefname{ALC@unique}{Step}{Steps}

\newenvironment{cpf}{\begin{trivlist} \item[] {\em Proof of claim.}}{\hspace*{\stretch{1}} $\diamond$ \end{trivlist}}

\newcommand{\intr}{\operatorname{intr}}

\title{Split cuts in the plane}
\author{Amitabh Basu\thanks{Department of Applied Mathematics and Statistics, Johns Hopkins University, Baltimore, MD, USA ({\tt basu.amitabh@jhu.edu}, {\tt hjiang32@jhu.edu}). Supported by the ONR Grant N000141812096, NSF Grant CCF2006587 and the AFOSR Grant FA95502010341.} 
\and Michele Conforti\thanks{Dipartimento di Matematica ``Tullio Levi-Civita'', Universit\`a degli Studi Padova, Italy ({\tt conforti@math.unipd.it}, {\tt disumma@math.unipd.it}). Supported by the PRIN grant 2015B5F27W and by a SID grant of the University of Padova.}
\and Marco Di Summa\footnotemark[3]
\and Hongyi Jiang\footnotemark[1]}

\begin{document}

\maketitle

\begin{abstract}
We provide a polynomial time cutting plane algorithm based on split cuts to solve integer programs in the plane. We also prove that the split closure of a polyhedron in the plane has polynomial size.
\end{abstract}

  \section{Introduction}

In this paper, we work in $\R^2$ and we always implicitly assume that all polyhedra, cones, half-planes, lines are rational. Given a polyhedron $P\subseteq\R^2$, we let $P_I:=\conv(P\cap \Z^2)$, where $\conv(\cdot)$ denotes the convex hull operator.

Given $\pi\in \Z^2\setminus\{0\}$ and $\pi_0\in \Z$,  let $H_0$ and $H_1$ be the half-planes defined by $\pi x\le \pi_0$ and $\pi x\ge \pi_0+1$, respectively. Given a polyhedron $P\subseteq\R^2$, we let $P_0:=P\cap H_0$, $P_1:=P\cap H_1$ and $P^{\pi,\pi_0}:=\conv(P_0\cup P_1)$. $P^{\pi,\pi_0}$ is a polyhedron that contains $P\cap \Z^2$. An inequality  $cx\le d$ is a {\em split inequality} (or {\em split cut}) for $P$ if there exist $\pi\in \Z^2\setminus\{0\}$ and $\pi_0\in \Z$ such that the inequality $cx\le d$ is valid for $P^{\pi,\pi_0}$. The vector $(\pi,\pi_0)$, or the set $H_0\cup H_1$, is a {\em split disjunction}, and we say that $cx\le d$ is a split inequality for $P$ with respect to $(\pi,\pi_0)$. The closed complement of a split disjunction, i.e., the set defined by $\pi_0\le\pi x\le\pi_0+1$, is called a {\em split set}.
If one of $P_0$, $P_1$ is empty, say  $P_1=\emptyset$,  the split inequality $\pi x\le \pi_0$  is called a {\em  Chv\'{a}tal inequality}.

 A Chv\'{a}tal inequality   $\pi x\le \pi_0$  where $\pi$ is a primitive vector (i.e., its coefficients are relatively prime) has the following geometric interpretation:  Let $z:=\max_{x\in P} \pi x$ and let $H$ be the half-plane defined by $\pi x\le z$. Then $P\subseteq H$ and $\pi_0=\lfloor z\rfloor$, because $\pi_0 \in \Z$ and $P_1=\emptyset$.  Since $\pi$ is a primitive vector,  $H_I$ is defined by the inequality  $\pi x\le \pi_0$. We will say that the inequality  $\pi x\le \pi_0$ defines the {\em Chv\'{a}tal strengthening} of the half-plane $H$. If $z=\pi_0$, then we say $H$ is {\em Chv\'{a}tal strengthened}.
  \medskip

An {\em inequality  description} of a polyhedron  $P\subseteq\R^2$ is   a system $Ax\le b$ such that $P=\{x\in \R^2:\;Ax\le b\}$, where $A\in \Z^{m\times 2}$ and $b\in \Z^m$ for some positive integer $m$. The {\em size} of the description of $P$, i.e., the number of bits needed to encode the linear system, is $O(m \log \|A\|_\infty+ m \log \|b\|_\infty)$. (Notation $\|\cdot\|_\infty$ indicates the infinity-norm of a vector or a matrix, i.e., the maximum absolute value of its entries.)
It follows from the above argument that when the coefficients in each row of $A$ are relatively prime integers, the inequalities defining $P$ are Chv\'{a}tal strengthened.
\medskip

Given  a polyhedron $P$, a  {\em cutting plane} or {\em cut}  is an inequality that defines a half-plane $H$ such  that $P\not\subseteq H$ but $P_I\subseteq H$. A  {\em cutting plane algorithm} is a procedure that, given a polyhedron $P\subseteq\R^2$ and a vector $c\in \Z^2$, solves the integer program $\max \{cx:x\in P\cap \Z^2\}$ by adding at each iteration a cut that eliminates an optimal vertex of the current continuous relaxation until integrality is achieved or infeasibility is proven.
\medskip

Integer programming in the plane is the problem $\max\{cx : Ax \leq b,\; x\in \Z^2\}$ where $c \in \Z^2, A \in \Z^{m\times 2}$ and $b \in \Z^m$. In \cref{sec:algorithm} we provide a cutting plane algorithm for this problem that uses split inequalities as cutting planes and such that the number of iterations (i.e., cutting planes computed) is $O(m (\log \|A\|_\infty)^2)$. (The derivation of every cutting plane can be carried out in polynomial time but involves a constant number of $\gcd$ computations.)
\medskip

We note that integer programming in the plane is well-studied and understood. In particular, given a polyhedron $P\subseteq\R^2$, Harvey \cite{harvey1999computing} gave an efficient procedure to produce an inequality description of $P_I$. Eisenbrand and Laue~\cite{Eisenbrand2d} gave an algorithm to solve the problem that makes $O(m + \max\{\log \|A\|_\infty, \log \|b\|_\infty, \log \|c\|_\infty\})$ arithmetic operations.

As split cuts are widely used in integer programming solvers, the scope of the present research is to prove that this class of integer programs can also be solved in polynomial time with a cutting plane algorithm based on split cuts (albeit not as efficiently as in \cite{Eisenbrand2d}).
\medskip

The second part of this paper deals with the complexity of the split closure of a polyhedron in the plane.
Given a polyhedron $P\subseteq\R^2$, the {\em split closure $P^{\rm split}$} of $P$ is defined as follows:
$$P^{\rm split}:=\bigcap_{\pi \in \Z^2\setminus\{0\},\pi_0\in \Z}P^{\pi,\pi_0}.$$
Cook, Kannan and Schrijver~\cite{cook-kannan-schrijver} proved that $P^{\rm split}$ is a polyhedron. Polyhedrality results for cutting plane closures, such as the above split closure result, have a long history in discrete optimization starting from the classical result that the {\em Chv\'atal closure} of a rational polytope (the intersection of all Chv\'atal inequalities) is polyhedral (see, e.g., Theorem 23.1 in~\cite{sch}), with several more recent results~\cite{MR2969261,dadush2011split,dadush2014chvatal,bhk2,zhu2020,pashkovich2019}, to sample a few. The {\em complexity} of cutting plane closures, i.e., the number of facets and the bit complexity of the facets, is relatively less understood. One of the most well-known results in this direction is due to  Bockmayr and Eisenbrand~\cite{bockmayrEisenbrand2001}, who showed that the complexity of the Chv\'atal closure of a rational polytope is polynomial in the description size of the polytope, if the dimension is a fixed constant (see \cref{thm:CG-closure} in \cref{sec:closure} below). It has long remained an open question whether the split closure is of polynomial complexity as well, even in the case of two dimensions. We settle this question in the affirmative in this paper; see \cref{thm:split-closure} in \cref{sec:closure}.

Finally, as again shown in~\cite{cook-kannan-schrijver}, if one defines $P^0:=P$ and recursively $P^i:=(P^{i-1})^{\rm split}$, then $P^t=P_I$ for some $t$. The {\em split rank} of $P$ is the smallest $t$ for which this occurs. It is a folklore result that if $P\subseteq \R^2$ is a polyhedron, its split rank is at most 2; we will observe in \cref{rem:int-hull} that this also follows from the arguments used in this paper. 

\paragraph{Some notation} Let $\dim(Q)$ denote the dimension of any polyhedron $Q$. A  polyhedron in $\R^2$ which is the intersection of two non parallel half-planes is a {\em full-dimensional translated pointed cone}. However, to simplify terminology we will often refer to such a polyhedron as a {\em translated cone}. Its unique vertex is  the {\em apex} of the cone. We will use the notation $(H_1,H_2):= H_1 \cap H_2$ to denote the translated cone formed by the intersection of two half-planes $H_1, H_2$. Given a half-plane $H$, we let $H^=$ denote its boundary.

\section{Tilt cuts and the clockwise algorithm}\label{sec:algorithm}

To simplify the presentation, throughout the paper the notions of facet and facet defining inequality of a polyhedron will be interchangeably used.

\begin{definition}(Tilt of a facet about a pivot with respect to a translated cone)\label{def:tilt}
Let $C = (H_1, H_2)$ be a translated cone with apex not in $\Z^2$, and assume that $H_1$ is Chv\'{a}tal strengthened. 

 Let $\widehat H$ be the line in $H_1$ parallel to $H_1^=$  and closest to $H_1^=$ such that $\widehat H\cap \Z^2\ne \emptyset$. Let $ p\in H_1^=\cap C\cap \Z^2$ and $q \in (H_1^=\setminus C)\cap \Z^2$ be the unique  points  such that the open line segment $(p,q)$ contains no integer point.  Let $x \in \widehat H\cap C\cap \Z^2$ and $y \in (\widehat H\setminus C)\cap \Z^2$ be the unique  points  such that the open line segment $(x,y)$ contains no integer point (possibly $x\in H_2^=$).

Two parallel sides of the parallelogram $P:=\conv(p,q,y,x)$ are contained in $H_1^=\cup\widehat H$. The other two sides of $P$ define a split disjunction in the following way. Let $W_0$, $W_1$ be the half-planes such that $W_0^=$ is the line containing $p$ and $x$, $W_1^=$ is the line containing $q$ and $y$, and $W_0\cap W_1=\emptyset$. As $P$ has integer vertices but contains no other integer point, $P$ has area 1 and $W_0$, $W_1$ define a split disjunction $(\pi,\pi_0)$.   

Let $F_1$ be the facet of $C$ induced by $H_1$. We now define the {\em tilt} $T$  of $F_1$ with {\em pivot $p$ with respect to $C$}. If $C\cap W_1=\emptyset$, then $(\pi,\pi_0)$ defines a Chv\'{a}tal cut for $C$ (as in \cref{fig:def1}(i)), and we let $T$ be this Chv\'{a}tal cut. Otherwise let $ x'\in W_1^=\cap C\cap \Z^2$ and $y' \in (W_1^=\setminus C)\cap \Z^2$ be the unique points such that the open line segment $(x',y')$ contains no integer point and let $q'$ be the point of intersection of $[x',y']$ and $H_2^=$ (possibly $q'=x'$), see \cref{fig:def1}(ii). We define $T$ as the split cut  for $C$ with respect to $(\pi,\pi_0)$ such that $T^=$ contains $p$ and $y'$.
(Note that in this case  $T$ is not the ``best'' split cut for $C$ with respect to $(\pi,\pi_0)$, as  it does  not define a facet of $\conv ((C \cap W_0)\cup (C \cap W_1))$.) 
\end{definition}

\begin{figure}
\includegraphics[width=0.75\textwidth]{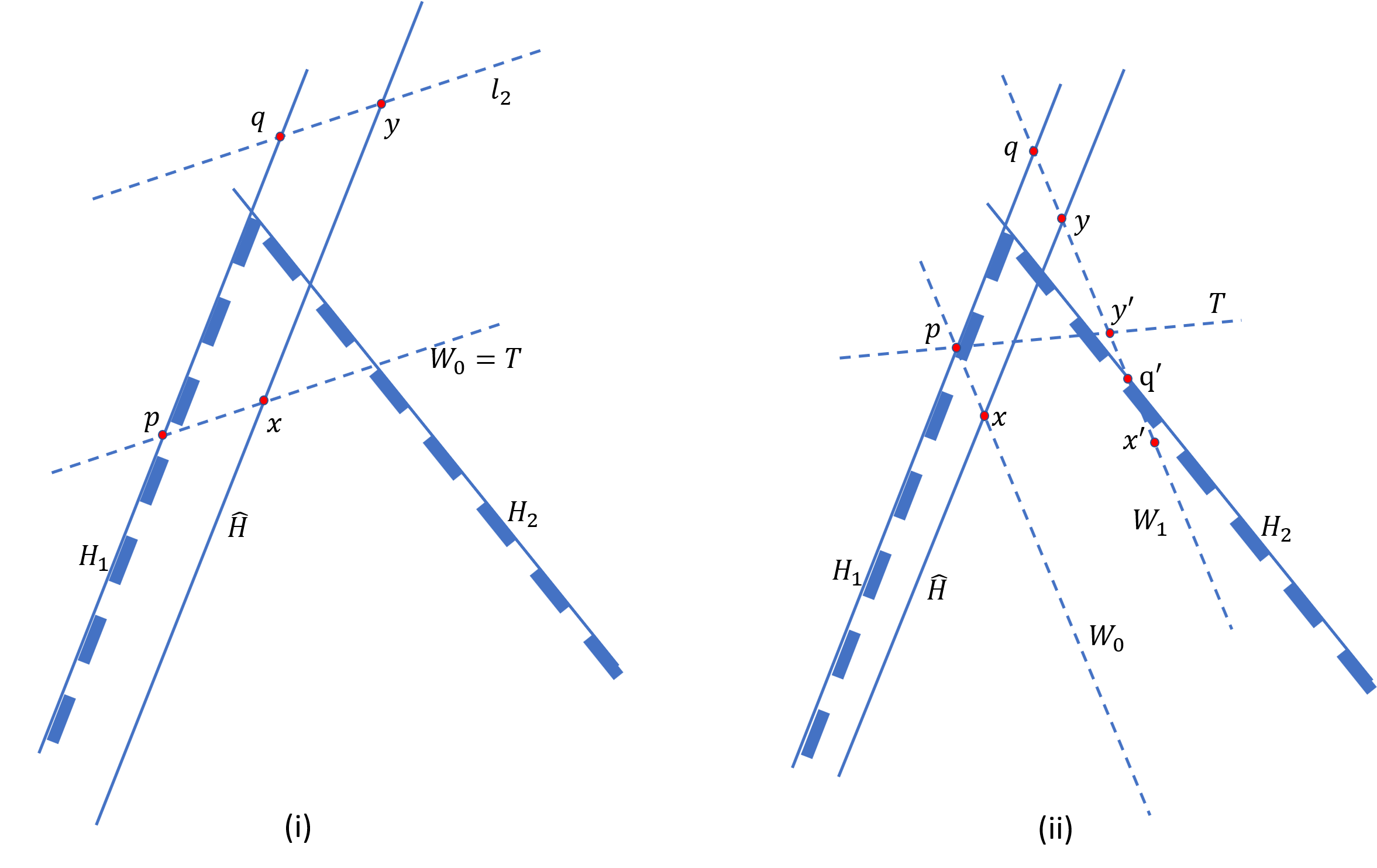}\centering
\caption{Illustration of \cref{def:tilt}.}
\label{fig:def1}
\end{figure}

In the next two lemmas we refer to the notation introduced in \cref{def:tilt}.

\begin{lemma}\label{lem:poly-q}
Let $ax \leq \beta$ and $dx \leq \delta$ be inequality descriptions of $H_1$ and $H_2$ respectively, where the coefficients of $a$ are relatively prime.
Then $dq - \delta \leq |a_1d_2 - a_2d_1|$.
\end{lemma}

\begin{proof}
Let $p$ be as in \cref{def:tilt}.
As $a_1, a_2$ are relatively prime, either $q = p + \begin{pmatrix} -a_2 \\ a_1\end{pmatrix}$ or $q = p + \begin{pmatrix} a_2 \\ -a_1\end{pmatrix}$. Taking an inner product with $d$, we obtain that $dq - dp \leq -a_2d_1 + a_1d_2$ in the first case and $dq - dp \leq a_2d_1 - a_1d_2$ in the second case.  
Since $p \in H_2$, $dp \leq \delta$. The result follows.\end{proof}

\begin{lemma}\label{le:tilt}  
Let $dx\le \delta$ be an inequality description of $H_2$.
\begin{enumerate}[(i)]
\item $T$  is  always Chv\'{a}tal strengthened. 
\item $T$ defines a facet of $C_I$ if and only if  $T$ is a Chv\'{a}tal cut for $C$.  
\item When $C\cap W_1\ne\emptyset$, we have $0<dy'-\delta\le \frac{dq-\delta}{2}$.
\end{enumerate}
\end{lemma}

\begin{proof}
(i) As  $p\in T^=\cap\Z^2$ and $T$ is a rational half-plane,  $T$ is  always Chv\'{a}tal strengthened.

(ii) Recall that $T$ is a Chv\'{a}tal cut for $C$ if and only if $C\cap W_1=\emptyset$. When $C\cap W_1=\emptyset$ we have $x\in T^=\cap C\cap \Z^2$, whereas when  $C\cap W_1\ne\emptyset$ we have $y'\in (T^=\setminus  C)\cap \Z^2$ and there is no integer point in the open segment $(p,y')$.
Since $T$ is a facet of $C_I$ if and only if $T^=\cap C$ contains an integer point different from $p$,  this happens if and only if $T$ is a Chv\'{a}tal cut.

(iii) When  $C\cap W_1\ne\emptyset$, we have that  $q,y,y'\in (W_1^=\cap \Z^2)\setminus  C$, $x' \in (W_1^=\cap \Z^2) \cap C$, while $q'\in W_1^=\cap  H_2^=$ and $(x',y')$ has no integer points. Therefore the length of $[y',q']$ is at most half the length of $[q,q']$. This implies that  $0<dy'-\delta\le \frac{dq-\delta}{2}$.
\end{proof}

 \begin{remark} The algorithm below uses the following fact:
If $P\subseteq \R^2$ is not full-dimensional, the integer program $\max \{cx: x \in P\cap \Z^2\}$ can be solved by applying one Chv\'{a}tal cut. Specifically, if $\dim(P) \leq 0$, the problem is trivial, and if $\dim(P) = 1$, with one cut we can certify infeasibility if $\aff(P)\cap \Z^2=\emptyset$ (where $\aff(P)$ denotes the affine hull of $P$). The cut certifying infeasibility has boundary parallel to $P$ in this case. Otherwise, if $\aff(P)\cap \Z^2\ne\emptyset$, the problem is unbounded if and only if $\max \{cx: x \in P\}=\infty$.  Finally, if $\aff(P)\cap \Z^2\ne\emptyset$ and $\max \{cx: x \in P\}$ is finite, the integer program is either infeasible or admits a finite  optimum: this can be determined by applying one  Chv\'{a}tal cut whose boundary is orthogonal to $P$. We note that, from standard results in integer programming, the Chv\'{a}tal cuts considered above have polynomial encoding length.
 \end{remark}

\begin{definition}\label{def:early-late}(Late facet and early facet)
Given an irredundant description $Ax\le b$ of a full-dimensional pointed polyhedron $P$ with $m$ facets, we denote by $F_i$ the facet of $P$ defined by the  the $i$th inequality $a^ix\le b_i$ of the system $Ax\le b$. Given a vector $c\in \Z^2\setminus\{0\}$ such that the linear program $\max\{cx:x\in P\}$ has finite optimum, and a specified optimal vertex $v$,  we assume that $a^1,\dots, a^m$ are ordered clockwise  so  that $v\in F_m\cap F_1$ and $c$ belongs to the cone generated by $a^m$ and $a^1$. 
We call $F_m$ the {\em late facet} and $F_1$ the {\em early facet} of $P$ with respect to $v$. This ordered pair defines the translated cone $(F_m,F_1)$ with apex $v$.
\end{definition}

\begin{algorithm}
\caption{The ``clockwise'' cutting plane algorithm}
\label{alg:buildtree}
\begin{algorithmic}[1]
\REQUIRE{A pointed polyhedron $P \subseteq \R^2$  and a vector $c \in \Z^2\setminus\{0\}$ such that $\max \{cx: x \in P\}$ is finite.}
\ENSURE{An optimal solution of the integer program $\max \{cx: x\in P\cap\Z^2$\} or a certificate of infeasibility.}
\STATE{Initialize $Q := P$.}
\STATE\label{step:dim}{If $\dim(Q)\le 1$, apply at most one Chv\'{a}tal cut to output INFEASIBLE or an optimal solution.}
\STATE{Else solve the linear program $\max \{cv: v \in Q\}$ and let $v^*$ be the optimal vertex. If
$v^* \in \Z^2$, STOP and output $v^*$.}
\STATE\label{step:cut}{If $v^* \not\in \Z^2$, number the facets of $Q$ in clockwise order $F_1, \ldots, F_{m^Q}$ so that $v^*\in (F_{m^Q},F_{1})$. If $F_{m^Q}$ is not Chv\'{a}tal strengthened, let $T$ be its Chv\'{a}tal strengthening. Otherwise
    let $T$ be the tilt of $F_{m^Q}$ with respect to $(F_{m^Q},F_{1})$.
 Update $Q := Q \cap T$ and go to \cref{step:dim}.}
\end{algorithmic}
\end{algorithm}

In Algorithm~\ref{alg:buildtree}, if $Q$ is full-dimensional and has two vertices that maximize $cx$ over $Q$, then $\text{arg}\max_{x\in Q}cx$ is a bounded facet of $Q$  (``optimal''  facet). We assume that:

\begin{assumption}\label{ass:optimal-vertex}
The optimal vertex $v^*$ is the first vertex encountered when traversing the optimal  facet  in clockwise order. 
\end{assumption}

Therefore if two vertices maximize $cx$, the optimal  facet is $F_1$ in our numbering.  With this convention, inequality $T$ computed in \cref{step:cut} at a given iteration will be tight for the vertex that is optimal at the successive iteration. Note that if $\text{arg}\max_{x\in Q}cx$ is an unbounded facet of $Q$ and the unique vertex is the first point encountered on this facet when traversing it clockwise, this facet is $F_1$ in our numbering, and if the unique vertex is the last point encountered on this facet when traversing it clockwise, this facet is $F_m$ in our numbering. See Figure~\ref{fig:assumption}.

\begin{figure}[htbp]
\includegraphics[width=0.75\textwidth]{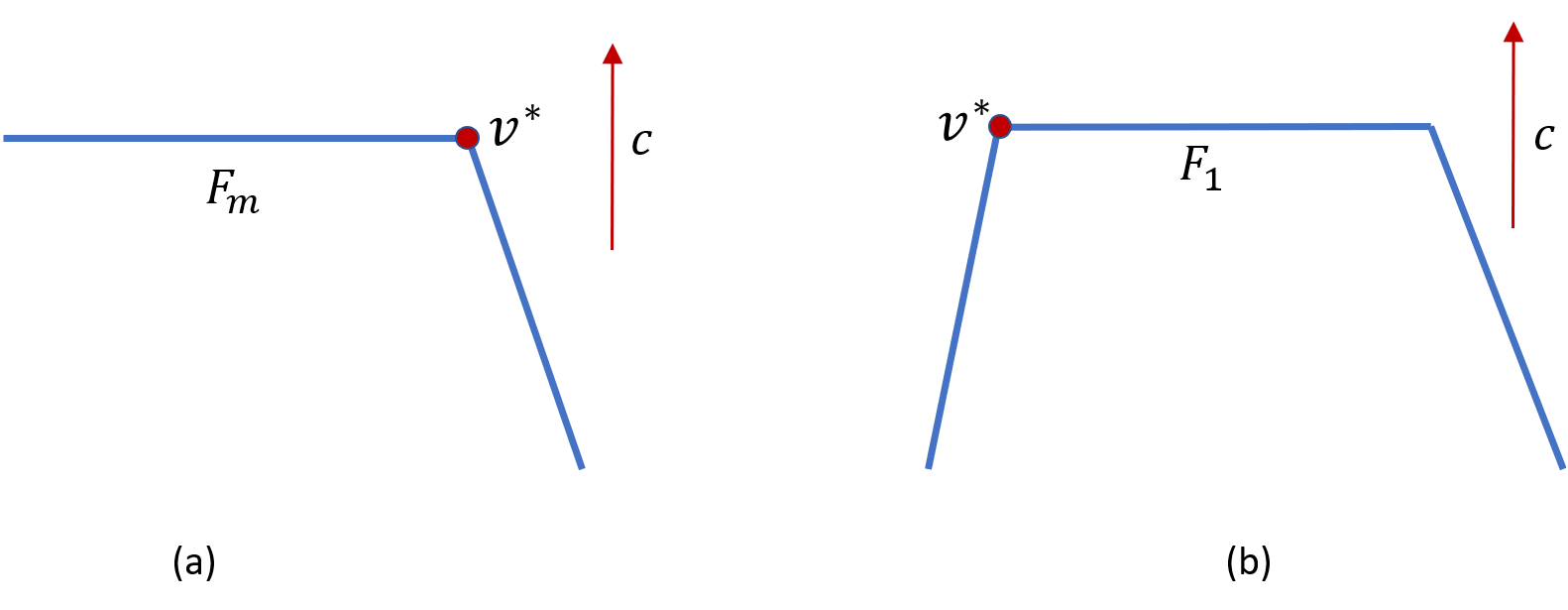}\centering
\caption{Illustration of \cref{ass:optimal-vertex}}
\label{fig:assumption}
\end{figure}

\begin{remark} A cutting plane algorithm typically works with a vertex solution, so it is natural to assume that $P$ is pointed. (The integer program can be solved with at most one Chv\'{a}tal cut  when  $P$ is a  polyhedron in the plane which is not pointed and it is known from standard results that this Chv\'{a}tal cut has a polynomial size).

If $P$ is a pointed polyhedron but $\max \{cx: x \in P\}$ is unbounded,  the integer program is either infeasible or unbounded. There are ways to overcome this, however it seems difficult to efficiently distinguish these two cases by only using cutting planes, even when $\dim(P)=2$.
\end{remark}

We will need the following theorem about the integer hulls of translated  cones in the plane.

\begin{theorem}[\cite{harvey1999computing}]\label{thm:2D-complexity}
Given a description $Ax\le b$ of a translated  cone $C \subseteq \R^2$,   $C_I$ has $O(\log \|A\|_\infty)$ facets.
  Furthermore, each facet of $C_I$ has a description $ax\le \beta$ where $\|a\|_\infty\le \|A\|_\infty$.
\end{theorem}

We also need the following lemma.

%
%
%

\begin{lemma}\label{le:sizeb}
Let $P$ be a pointed polyhedron in $\R^n$ such that $P_I\ne \emptyset$. Let $u$ be the largest infinity norm of a vertex of $P$ or $P_I$.  Let $ax\le \beta$ be an inequality which is valid for $P_I$ but is not valid for $P$. Then $|\beta|  \le nu\|a\|_\infty$. 
\end{lemma}

\begin{proof}  As $P_I\ne \emptyset$, by Meyer's theorem (see e.g. Theorem 4.30 in~\cite{conforti2014integer}) $P$ and $P_I$ have the same recession cone. Therefore $\max_{x\in P}ax$ is finite and is larger than $\beta$, because  $ax\le\beta$ is not valid for $P$.
Since finite maxima are attained at vertices, we have that  $-nu\|a\|_\infty\le \max_{x\in P_I}ax\le \beta< \max_{x\in P}ax\le nu\|a\|_\infty$, which proves the lemma.
\end{proof}

Next, we will prove that when the relaxation is a translated cone, the clockwise algorithm finds the solution in polynomial time. An important point of the proof is that when the tilt is not a  Chv\'{a}tal cut, the pivot remains unchanged. Otherwise, a new facet of the integer hull of the translated cone is obtained when the tilt is a Chv\'{a}tal cut (see~\cref{le:tilt} (ii)). Since one has only a polynomial number of facets of the integer hull (see~\cref{thm:2D-complexity}), it suffices to show that only a polynomial number of tilts are made about any one fixed pivot. This is achieved by appealing to Lemmas~\ref{lem:poly-q} and~\ref{le:tilt}. We formalize this below.

\begin{theorem}\label{thm:algCone}
Let $Ax\le b$ be a description of a translated  cone $C\subseteq\R^2$ and let $c\in \Z^2\setminus\{0\}$ be such that $\max\{cx:x\in C\}$ is finite. Then the clockwise algorithm  solves the integer program $\max\{cx: x\in C\cap \Z^2\}$ in  $O((\log \|A\|_\infty)^2)$  iterations. Furthermore, there is a polynomial function $f(\cdot,\cdot)$ (independent of the data) such that every cut computed by the algorithm admits a description $ax\le \beta$ where  $\|a\|_\infty\le \|A\|_\infty$ and $|\beta|\le f(\|A\|_\infty,\|b\|_\infty)$.
\end{theorem}

\begin{proof}
We use the same notation as in \cref{def:tilt}  and the fact that since $C$ is a translated  cone, if $p\in \Z^2$ is a pivot element of a cut computed by the algorithm, then $p\in C$.

Let $T_i$ be the cutting plane produced by the clockwise algorithm at iteration $i$, where we assume that $T_0$ is the Chv\'{a}tal strengthening of $H_1$. 
\medskip

\noindent{\sc Claim 1.} 
{\em If $T_i$ defines the early facet of $C \cap T_0 \cap \dots \cap T_i$, the clockwise algorithm computes an optimal solution in iteration $i+1$.}

\begin{cpf} In this case  $(T_{i-1}, T_{i})$ is the new translated cone whose apex is the pivot element $p_i$ of iteration $i$. As $p_i\in C\cap \Z^2$, at iteration $i+1$ the algorithm determines that $p_i$ is an optimal solution.
\end{cpf}

$T_i$ is the tilt (with pivot element $p_i$) with respect to the translated cone $(T_{i-1},H_2)$. Also recall that by \cref{le:tilt} (i), $T_i$ is Chv\'{a}tal strengthened. This fact will be important because we will work with the translated cone $(T_i,H_2)$ below and use notions from \cref{def:tilt} and results based on these notions, which assume that the facet $H_1 = T_i$ of the translated cone is Chv\'atal strengthened.\medskip

\noindent{\sc Claim 2.} 
{\em There is a polynomial function $f(\cdot,\cdot)$ such that $T_i$ admits a description $ax\le \beta$ where  $\|a\|_\infty\le \|A\|_\infty$ and $|\beta|\le f(\|A\|_\infty,\|b\|_\infty)$.}

\begin{cpf}  By induction on $i$. The base case $i=0$ is trivial.

We first show that $T_i$ admits a description $ax\le \beta$ where  $\|a\|_\infty\le \|A\|_\infty$.
If $T_i$ defines a Chv\'{a}tal cut with respect to the translated cone $(T_{i-1},H_2)$, as  by induction $T_{i-1}$ satisfies the claim, we are done by \cref{le:tilt} (ii) and \cref{thm:2D-complexity}.

  So we assume that $(T_{i-1},H_2)\cap W_1\ne \emptyset$ (where $W_1$ is as in \cref{def:tilt} with respect to $(T_{i-1},H_2)$).  Consider  the translated cone $(T_{i-1},H_2')$, where $H_2'$ is the translation of $H_2$ through $y'$ (see Figure 1(ii)).   As $T_i$ is the tilt with respect to $(T_{i-1},H_2)$,  we have that $p_i, y'\in T_{i}^=\cap \Z^2$. Furthermore, $T_{i-1}$ satisfies the claim by induction. It follows that $T_i$ is a facet of the integer hull of $(T_{i-1},H_2')$ and, by \cref{thm:2D-complexity}, $T_i$ admits a description $ax\le \beta$ where  $\|a\|_\infty\le \|A\|_\infty$.     
  
  Let $u$ be the largest infinity norm of a vertex of $C$ or $C_I$. Then $u$  is bounded by a polynomial function of $\|A\|_\infty$ and $\|b\|_\infty$ (see, e.g., \cite[Theorems 10.2 and 17.1]{sch}). Therefore, by \cref{le:sizeb}, there is a polynomial function $f(\cdot,\cdot)$ such that $|\beta|\le f(\|A\|_\infty,\|b\|_\infty)$. This completes the proof of the claim.
  \end{cpf}

  We finally show that in $O((\log \|A\|_\infty)^2)$  iterations  the clockwise algorithm  finds an optimal solution to the program $\max \{cx: x\in C\cap \Z^2\}$.   By Claim 1 if the cut $T_i$ becomes the early facet in iteration $i$, then the algorithm finds an optimal solution in iteration $i+1$. By Claim 2  all cuts $T_1, \ldots, T_i$ admit a description $ax\le \beta$ where  $\|a\|_\infty\le \|A\|_\infty$ and $|\beta|\le f(\|A\|_\infty,\|b\|_\infty)$.
  Therefore, by \cref{thm:2D-complexity}, it suffices to show that within at most $O(\log \|A\|_\infty)$ iterations beyond any particular iteration $i$, the algorithm either finds an optimal solution or computes the facet adjacent to $T_i$ of the integer hull of the translated cone $(T_i,H_2)$. Note that, as all pivot elements are in $C$, this is also a facet of $C_I$.
  
 By \cref{le:tilt} (ii), the facet adjacent to $T_i$ of the integer hull of the translated cone $(T_i,H_2)$ is obtained when $W_1\cap (T_i,H_2)=\emptyset$,  i.e., when $T_{i+1}$ is a Chv\'{a}tal cut (this $W_1$ is as in \cref{def:tilt} with respect to the translated cone $(T_i,H_2)$). If $W_1\cap (T_i, H_2)\ne\emptyset$,  by \cref{le:tilt} (iii), we have that $0<dy'-\delta\le \frac{dq-\delta}{2}$, where $dx\le \delta$ is the inequality defining $H_2$ in the description of $(T_i, H_2)$.  We now observe that the new translated cone to be processed is $(T_{i+1},H_2)$ with apex $y'$. Moreover, the pivot remains $p$ for the next iteration. Further, it is important that the facet $H_2$ has not changed. Thus, one may iterate this argument and since $dy'-\delta\in\Z$ at every iteration, by \cref{lem:poly-q}, after at most $O(\log\|A\|_\infty)$ iterations the algorithm will produce a Chv\'{a}tal cut.  
  \end{proof}

\begin{cor}\label{cor:encoding} Any cut derived during the execution of the clockwise cutting plane algorithm on any pointed, full-dimensional polyhedron $P$ admits a description $ax\le \beta$ where  $\|a\|_\infty\le \|A\|_\infty$ and $|\beta|\le f(\|A\|_\infty,\|b\|_\infty)$, where $f(\cdot, \cdot)$ is the function from \cref{thm:algCone}.
\end{cor}

\begin{proof} 
The argument in Claim 2 of the proof of Theorem~\ref{thm:algCone} also applies here to show that $\|a\|_\infty\le \|A\|_\infty$, since we introduce new cutting planes by processing translated cones formed by original facet defining inequalities of $P$ or the previous cuts. 

To argue for the right hand side $\beta$, we consider two cases: when $P_I \neq \emptyset$ and when $P_I = \emptyset$. In the first case, the result follows from Lemma~\ref{le:sizeb}. In the second case, we again break the proof into two cases: when $P$ is bounded and when $P$ is unbounded. When $P$ is bounded, note that all the cuts except for the last one must be valid for at least one vertex of $P$ (otherwise, we would have proved infeasibility). For any such cut, let $w$ be the vertex that is valid and let $w'$ be a vertex that is cut off (which must exist because no cut in the algorithm is valid for $P$). Then, $-2\|w\|_\infty\|a\|_\infty\le aw\le \beta< \max_{x\in P}ax\le 2\|w'\|_\infty\|a\|_\infty$. This proves the claim for all cuts except for the last one. The last cut is obtained from a translated cone defined by original inequalities and/or previous cuts. Theorem~\ref{thm:algCone} now applies. 

We now tackle the case when $P$ is unbounded and $P_I = \emptyset$. This implies that $P$ has two unbounded facets. Let the half-planes corresponding to these unbounded facets be $H_1$ and $H_2$. We must have that $H_1 \cap H_2$ is a split set. Let the cuts introduced by the algorithms be $T_1, \ldots, T_k$. While at least one vertex of $P$ survives, the argument from above can be used to bound $|\beta|$. Thus, let $i \in \{1, \ldots, k\}$ be the smallest index such that $T_i$ cuts off every vertex of $P$.  $T_i$ must be derived from a translated cone using original inequalities and/or previous cuts and Theorem~\ref{thm:algCone} applies. If $i=k$, then we are done. Otherwise, after iteration $i$ we have a polyhedron with three facets $H_1$, $T_i$ and $H_2$, in clockwise order. If the new optimal vertex is at the intersection of $H_1$ and $T_i$, then the Chv\'atal strengthening of $H_1$ gives an empty polyhedron and the algorithm terminates. Else, if the new optimal vertex is the intersection of $T_i$ and $H_2$, then the new cutting plane $T_{i+1}$ will be a tilt of $T_i$ and will therefore also cut off every vertex of $P$. Repeating this argument, it follows that for all $i \leq j \leq k$, the half-plane $T_j$ cuts off every vertex of $P$, and $T_j \cap P = T_j \cap H_1 \cap H_2$. Moreover, all the cuts $T_i, \ldots, T_{k-1}$ are valid for the integer hull of the translated cone formed by $T_i$ and $H_2$, and Theorem~\ref{thm:algCone} applies. Finally, $T_k$ is either a Chv\'atal cut obtained by strengthening $H_1$, or is a cut valid for the translated cone formed by $T_{k-1}$ and $H_2$. In either case, we are done by previous arguments.
\end{proof}

We now extend~\cref{thm:algCone} to handle the case of a general pointed polyhedron $P\subseteq\R^2$.

\begin{definition}\label{def:facet-ordering}(Facet ordering)
We let $Q_i$ be the polyhedron computed at the beginning of  iteration $i$ of the clockwise algorithm and $T_i$ be the cutting plane  computed at iteration $i$. We start our iterations at $i=0, 1, \ldots$, so $Q_0=P$ and $Q_i=Q_{i-1}\cap T_{i-1}$. When $Q_i$ is full-dimensional, we let $F_{i,1},\dots, F_{i,m_i}$  be the facets of  $Q_i$ so that the optimal vertex of  $Q_i$ is the apex of $(F_{i,m_i},F_{i,1})$ and $T_i$ is either the Chv\'{a}tal strengthening or the tilt of $F_{i,m_i}$ with respect to  $(F_{i,m_i},F_{i,1})$.
\end{definition}

Note that, when $Q_i$ is full-dimensional, $T_i$ is either the early or the late facet of $Q_{i+1}$, as $T_i$ defines the optimal vertex chosen by the algorithm.

\begin{definition}\label{def:angle}(Potentially late and potentially early facets)
Given vectors  $a$, $b$, we define $\angle(a,b)$ as the clockwise angle between $a$ and $b$, starting from $a$.
When $Q_i$ is full-dimensional, let $a_{i,1},\dots, a_{i,m_i}$ be the normals of $F_{i,1},\dots, F_{i,m_i}$ (as defined in \cref{def:facet-ordering}). Then $\angle(a_{i,m_i},c)<180^\circ$ and  $\angle(c,a_{i,1})<180^\circ$. We define a facet $F$ of $Q_i$ with normal $a$ {\em potentially late} if either $F=F_{i,m_i}$ or $0<\angle(a,c)<180^\circ$ and {\em potentially early} if  either  $F=F_{i,1}$ or $0<\angle(c,a)\leq180^\circ$. 
 Note that if a facet of $Q_i$ 
 satisfies $\angle(c,a)=180^\circ$, then it cannot define the optimal vertex of $Q_j$, $j>i$.
\end{definition} 

%
 
 \begin{lemma}\label{le:early-late}
Consider iterations $i$ and $j > i$ and let $Q_i$ and $Q_j$ be the full-dimensional polyhedra defined in Definition \ref{def:facet-ordering} at these iterations. If $F$ is a potentially early facet of $Q_i$, then it cannot become a potentially late facet of $Q_j$ and  if $F$ is a potentially late facet of $Q_i$, then it cannot become a potentially early facet of $Q_j$.
 \end{lemma}

\begin{proof} Let $a$ be the normal of $F$.  The result is obvious when $\angle(a,c)>0$ and $\angle(c,a)>0$.  
If $\angle(a,c)=0$, i.e., $F=\text{arg}\max_{x\in Q_i}cx$, then $F$ remains potentially late or potentially early, by the choice of the optimal vertex; see \cref{ass:optimal-vertex}.
\end{proof} 

\begin{definition}(Families of facets and tilts)
Given a full-dimensional pointed polyhedron $P=Q_0 \subseteq \R^2$  and an objective vector $c \in \Z^2\setminus\{0\}$ such that $\max \{cx: x \in P\}$ is finite, 
 let $F_{0,1},\dots, F_{0,k}$ be the potentially early facets of $Q_0$ and $F_{0,k+1},\dots, F_{0,m_0}$ be the potentially late facets of $Q_0$. We say that facet $F_{0,\ell}$ of $P$ belongs to {\em family} $\ell$ and we recursively assign a cut $T_i$ produced at iteration $i$ of the algorithm to the family of the late inequality that is used to produce $T_i$. We finally say that family $\ell$ is  {\em extinct} at iteration $k$ if no  facet of $Q_k$ belongs to family $\ell$.
\end{definition}

\begin{remark}\label{rem:late-families}
By \cref{le:early-late}, no facet  that is potentially early can become potentially late and vice versa; therefore, all  cuts produced by the clockwise algorithm belong to the $m-k$ families associated with the potentially late facets of $Q_0$ (assuming the input to the algorithm is full-dimensional; otherwise, the algorithm terminates in at most two iterations ---see \cref{step:dim} in Algorithm \ref{alg:buildtree}). 
\end{remark}

\begin{figure}
\includegraphics[width=0.75\textwidth]{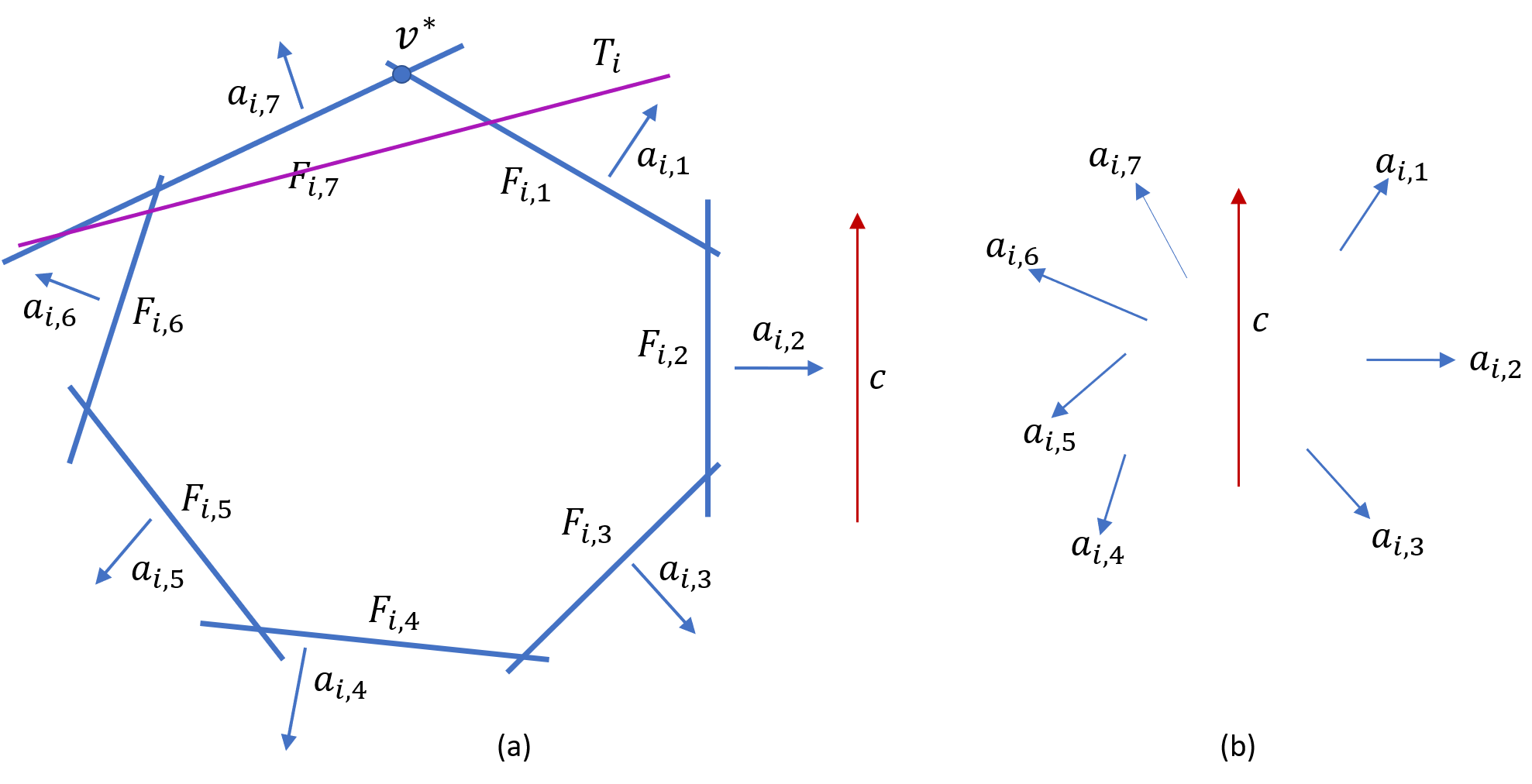}\centering
\caption{Illustration of the proof of \cref{thm:mainalgo}. In (a), $m_i=7$, $F_{i,1}$ is the early facet, and $F_{i,7}$ is the late facet as defined in \cref{def:early-late}. In (b), the left-hand side normals correspond to potentially late families, while those at right-hand side belong to potentially early families. Thus, $E_i = 3$ ($L_i$ will depend on the original set of facets).}
\label{fig:thm217a}
\end{figure}

\begin{figure}
\includegraphics[width=0.5\textwidth]{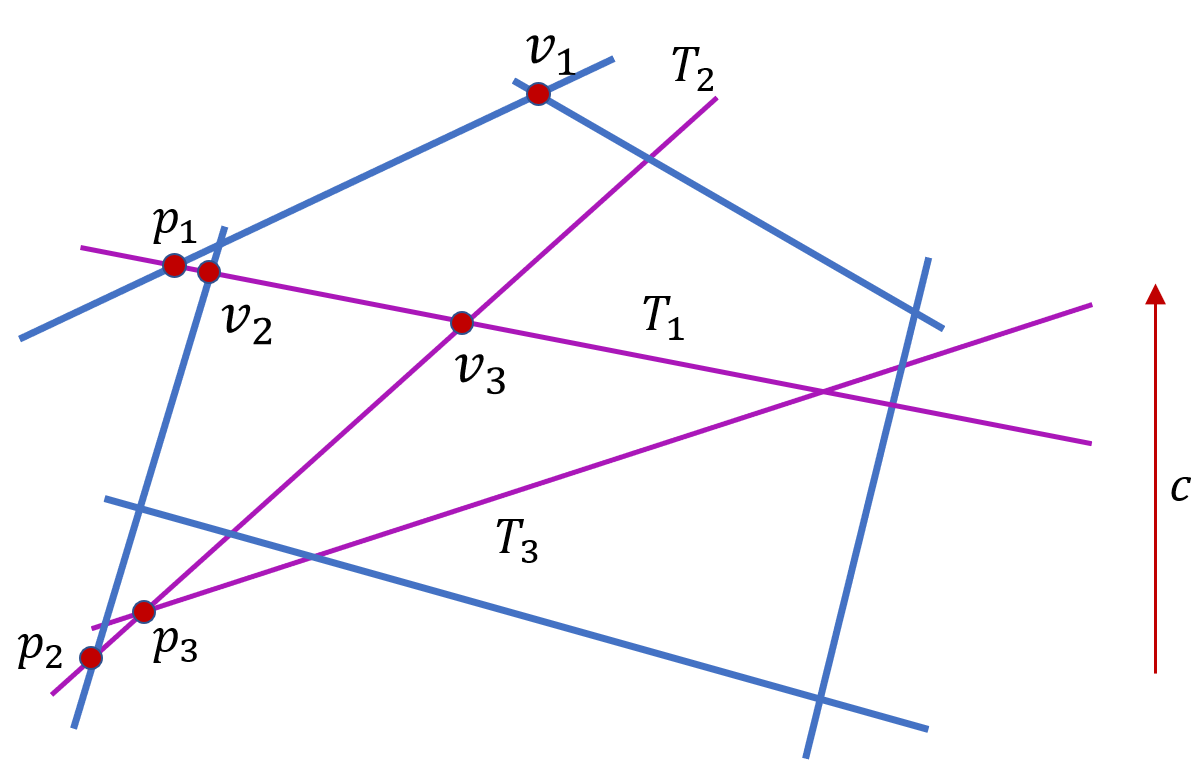}\centering
\caption{Illustration of an instance $P$ with $3$ iterations of the algorithm producing cuts $T_1$, $T_2$ and $T_3$ with pivots $p_1$, $p_2$ and $p_3$ respectively. The consecutive LP optimal solutions are labeled as $v_1$, $v_2$ and $v_3$. }
\label{fig:thm217b}
\end{figure}

\begin{theorem}\label{thm:mainalgo}
Let $Ax\le b$ be a description of  a pointed polyhedron $P\subseteq\R^2$ with $m$ facets, and $c\in\Z^2\setminus\{0\}$ be such that $\max \{cx: x \in P\}$ is finite. Then the clockwise algorithm solves the integer program $\max \{cx: x \in P\cap \Z^2\}$ in  $O(m(\log \|A\|_\infty)^2)$ iterations.
\end{theorem}

\begin{proof}
We refer to the definitions of $Q_i$ and $T_i$, and when $Q_i$ is full-dimensional, to the definitions of $F_{i,1}, \dots, F_{i, m_i}$ with corresponding normals $a_{i,1}, \dots, a_{i,m_i}$ (see \cref{def:facet-ordering} and \cref{def:angle}).
In this case, we assume $F_{i,1},\dots, F_{i,k_i}$ are potentially early and $F_{i,k_i+1},\dots, F_{i,m_i}$ are potentially late.
Moreover, let $E_i$ be the number of facets of $Q_i$ that are potentially early (i.e., $E_i=k_i$) and let $L_i$ be the number of families that are not extinct at iteration $i$ and such that the last inequality added to the family is potentially late. See Figure~\ref{fig:thm217a} for an illustration of these notations. Figure~\ref{fig:thm217b} illustrates some iterations of a potential run of the algorithm.
 
By \cref{thm:algCone}, there exists a function $z\mapsto g(z)$, where $g\in O((\log z)^2)$, such that the clockwise algorithm applied to any translated  cone with description $A'x\le b'$ terminates in $g(\|A'\|_\infty)$ iterations. Define $t:=g(\|A\|_\infty)$.
\medskip

\noindent{\sc Claim.} {\em Assume $\dim(Q_i) = 2$. Let $j$ be the largest natural number such that at iteration $i+j$, $\dim(Q_{i+j}) = 2$, $F_{i+j,m_{i+j}}$ is in the same family as $F_{i,m_i}$, and $F_{i+j,1}$ and $F_{i,1}$ are both defined by the same normal. 
Then
\begin{enumerate}
\item $j \leq t$, and
\item either $\dim(Q_{i+j+1}) \leq 1$, or the algorithm terminates at iteration $i+j+1$, or $E_{i+j+1}+2L_{i+j+1}\le E_{i}+2L_{i}-1$.
\end{enumerate}}

\begin{cpf} 1.\ follows from \cref{thm:algCone} and \cref{cor:encoding}, after observing that during iterations $i,\dots,i+j$ the algorithm computes the same cuts as those that it would compute if the polyhedron at iteration $i$ was the translated cone $(F_{i,m_i},F_{i,1})$.

We now prove 2.
Suppose $\dim(Q_{i+j+1}) = 2.$ We will establish that either the algorithm terminates at iteration $i+j+1$ or $E_{i+j+1}+2L_{i+j+1}\le E_{i}+2L_{i}-1$. Let $(F_{i+j+1,m_{i+j+1}},F_{i+j+1,1})$ be the translated cone at iteration $i+j+1$. Recall that, by the choice of the optimal vertex, $T_{i+j}$ must be either $F_{i+j+1,m_{i+j+1}}$ or $F_{i+j+1,1}$. We distinguish two cases. \smallskip

\noindent{\bf Case 1} Assume $T_{i+j} = F_{i+j+1,m_{i+j+1}}$.  Since $T_{i+j}$ is in the same family as $F_{i+j,m_{i+j}}$, which is in the same family as $F_{i,m_i}$, by definition of $q$ we must have that $F_{i,1}$ is distinct from $F_{i+j+1,1}$. Then $F_{i,1}$ is a redundant inequality for $Q_{i+j+1}$. Therefore, by \cref{le:early-late}, $E_{i+j+1} \leq E_{i} - 1$ and $L_{i+j+1} = L_i$. Thus, $E_{i+j+1}+2L_{i+j+1}\le E_{i}+2L_{i}-1$.\smallskip

\noindent{\bf Case 2}  Assume $T_{i+j}=F_{i+j+1,1}$. Let $p$ be the pivot element of the translated cone $(F_{i+j,m_{i+j}},F_{i+j,1})$. If $p$ is feasible, i.e., $p \in P \cap \Z^2$, then $p$ will be the optimal integral vertex in the iteration $i+j+1$ and the algorithm terminates. Else $p$ is infeasible. This means that $p$ must violate the inequality defining $F_{i+j+1,m_{i+j+1}}$. Consider the facet $F$ of $P$ that is the original facet of $P$ from the same family as $T_{i+j}$. We will now show that $F$ and all the inequalities in this family except for $T_{i+j}$ are redundant for $(F_{i+j+1,m_{i+j+1}},T_{i+j})$ and therefore for $Q_{i+j+1} \subseteq (F_{i+j+1,m_{i+j+1}},T_{i+j})$. Since we have processed this family during the algorithm, there must have been an optimal vertex defined by $(F, F')$ for some inequality $F'$ that is facet defining at some point during the algorithm. Let $Q'$ be the polyhedron computed as in Definition \ref{def:facet-ordering} at the iteration of the algorithm when the vertex $v$ defined by $(F, F')$ was optimal.

Since $F_{i+j+1,m_{i+j+1}}$ is not redundant for $Q'$ (as it is not redundant for $Q_{i+j+1}$ which is a subset of $Q'$), and the inequality defining $F_{i+j+1,m_{i+j+1}}$ is valid for the optimal vertex $v$ of $Q'$, its normal vector cannot be contained in the cone generated by the normals of $F$ and $F'$. Since the vertex defined by $(F, F')$ was optimal for the relaxation $Q'$, $c$ is contained in the cone generated by the normals of $F$ and $F'$. Thus, the normal of $F_{i+j+1,m_{i+j+1}}$ cannot be contained in the cone generated by the normal of $F$ and $c$. This means that the normal of $F$ is contained in the cone between the normal of $F_{i+j+1,m_{i+j+1}}$ and $c$, since both $F$ and $F_{i+j+1,m_{i+j+1}}$ are late facets at some time during the algorithm. Moreover, the normals of the inequalities in the family of $F$ are contained in the cone generated by the normal of $F$ and $T_{i+j}$, and therefore, in the cone generated by the normals of $F_{i+j+1,m_{i+j+1}}$ and $T_{i+j}$. Since our current optimal vertex is defined by $F_{i+j+1,m_{i+j+1}}$ and $T_{i+j}$, all these inequalities from the family must be redundant for $(F_{i+j+1,m_{i+j+1}},T_{i+j})$.

Thus, we have established that $F$ and all the inequalities in its family except for $T_{i+j}$ are redundant for $Q_{i+j+1}$. Since $T_{i+j}$ is from the same family and is early at iteration $i+j+1$, we must have $L_{i+j+1} \leq L_i - 1$ by \cref{le:early-late}. Moreover, $T_{i+j}$ is the only new early facet, and therefore, $E_{i+j} \leq E_i + 1$ by \cref{le:early-late}. Thus, $E_{i+j+1}+2L_{i+j+1}\le E_{i}+2L_{i}-1$. This completes the proof of the claim.
\end{cpf}

By the above claim, in at most $O(\log \| A\|_\infty^2)$ iterations after iteration $i$, either the algorithm terminates or the number $E_i + 2L_i$ must decrease by at least 1. Since the maximum value of $E_i + 2L_i$ is at most $2m$, we have the result.
\end{proof}

\begin{remark}
The upper bound on the number of iterations given in the above theorem does not depend on $c$.
\end{remark}

\begin{remark}
By \cref{le:tilt},  when $C\cap W_1\ne \emptyset$, the tilt $T$ produced by the algorithm may not be a facet of the split closure of the translated cone $C$. 
However, 
this property is crucial for the convergence of the algorithm.  Indeed, if the ``best cut'' is used, the algorithm may not converge as shown by the following example.

Define $p_0:=3$ and $p_{i+1}:= 2p_i-2$ for all integers $i\ge1$. Given $i\geq 0$, consider the following integer program, which we denote by $P_i$:
\begin{align}
     \max ~&x_2\\
    \mbox{s.t.} ~ &x_1\leq 4\label{eq:1}\\
    &(2p_i-1)x_1-(4p_i-4)x_2\geq 0\label{eq:2}\\
    &5x_1-8x_2 \geq 0\label{eq:3}\\
    &x_1,~x_2\in \Z
\end{align}
(Note that for $i=0$ the inequalities \cref{eq:2} and \cref{eq:3} coincide.)
The optimal solution of the continuous relaxation is $\left(4, \frac{2p_i-1}{p_i-1}\right)\notin \Z^2$, which is the unique point satisfying both constraints \cref{eq:1} and \cref{eq:2} at equality.

We will use the same notation as in \cref{def:tilt}, where $H_1$ and $H_2$ are the two half-planes defined by \cref{eq:2} and \cref{eq:1}, respectively. Since $(2p_i-1)$ and $(4p_i-4)$ are coprime numbers, $\hat{H}$ is defined by the equation $(2p_i-1)x_1-(4p_i-4)x_2=1$. Combined with the fact that the pivot is $p=(0, 0)$, this implies that $q = (4p_i-4,2p_i-1)$. 

We claim that $x=(-2p_i + 3, -p_i+1)$ and $y=(2p_i-1,~p_i)$. This follows from the following three observations: (i) these two points are integer and belong to $\hat H$; (ii) $-2p_i+3<4<2p_i-1$ (because this is true for $i=0$ and $p_i$ increases as $i$ increases); (iii) $y-x=q-p$.

It follows that $W_1^=$ is defined by the equation $(p_i-1)x_1-(2p_i-3)x_2 = 1$. This line intersects the edge $\left\{(x_1, x_2): x_1=4,~x_2\leq \frac{2p_i-1}{p_i-1}\right\}$ of the continuous relaxation at $q'= \left(4,\frac{4p_i-5}{2p_i-3}\right)$ (note that $\frac{4p_i-5}{2p_i-3}<\frac{2p_i-1}{p_i-1}$). Thus the strongest cut would be $(4p_i-5)x_1-(8p_i-12)x_2\ge0$. Since $p_{i+1}=2p_{i}-2$, the cut can be written as $(2p_{i+1}-1)x_1-(4p_{i+1}-4)x_2\geq 0$. When we add this cut to the continuous relaxation, \cref{eq:2} becomes redundant and we obtain problem $P_{i+1}$. Then this procedure never terminates.
\end{remark}

\section{Polynomiality of the split closure}\label{sec:closure}

In this section we prove the following result:

\begin{theorem}\label{thm:split-closure}
Let $Ax\le b$ be a description of a polyhedron $P\subseteq\R^2$ consisting of $m$ inequalities. Then the split closure of $P$ admits a description whose size is polynomial in $m$, $\log\|A\|_\infty$ and $\log\|b\|_\infty$.
\end{theorem}

We will make use of the following result, which holds in any fixed dimension.

\begin{theorem}\label{thm:CG-closure}
\cite{bockmayrEisenbrand2001}
Let $d\ge1$ be a fixed integer and let $Ax\le b$ be a description of a polyhedron $P\subseteq\R^d$ consisting of $m$ inequalities. Then the Chv\'atal closure of $P$ admits a description whose size is polynomial in $m$, $\log\|A\|_\infty$ and $\log\|b\|_\infty$.
\end{theorem}

Because of \cref{thm:CG-closure}, in order to prove \cref{thm:split-closure} it is sufficient to show that the intersection of all the split cuts for $P$ that are not Chv\'atal cuts is a polyhedron that admits a description of polynomial size.\medskip

We now start the proof of \cref{thm:split-closure}.
We can assume that $P\subseteq\R^2$ is pointed, as otherwise it is immediate to see that the split closure of $P$ is $P_I$ and is defined by at most two inequalities.
The following result holds in any dimension.

\begin{lemma}[\cite{DBLP:journals/mp/AndersenCL05}; see also {\cite[Corollary 5.7]{conforti2014integer}}]\label{lem:corner}
The split closure of $P$ is the intersection of the split closures of all the corner relaxations of $P$ (i.e., relaxations obtained by selecting a feasible or infeasible basis of the system $Ax\le b$).
\end{lemma}

Since there are at most $m\choose 2$ corner relaxations of $P$ (i.e., bases of $Ax\le b$), because of \cref{lem:corner} in the following we will work with a corner relaxation of $P$, which we denote by $C$. Thus $C$ is a full-dimensional translated pointed cone. We denote its apex by $v$.\smallskip

\begin{definition}\label{def:effective}(Effective split sets and effective split cuts)
We say that a split set $S$ is {\em effective} for $C$ if $v$ lies in its interior; note that this happens if and only if there is a split cut for $C$ derived from $S$ that cuts off $v$. Such a split cut will also be called effective.
\end{definition}

Since $C$ is a translated cone, for every effective split disjunction $(\pi,\pi_0)$ we have $C^{\pi,\pi_0}=C\cap H$ for a unique split cut $H$ derived from this disjunction. In the following, whenever we say ``{\em the} split cut derived from a given disjunction'' we refer to this specific split cut. 
Note that when the boundary of an effective split set $S$ intersects the facets of $C$ in precisely two points, the split cut derived from $S$ is delimited by the line containing these two points, while when the boundary of $S$ intersects the facets of $C$ in a single point, the line delimiting the split cut derived from $S$ contains this point and is parallel to the lineality space of $S$. (In the latter case, the split cut is necessarily a Chv\'atal cut.)

In the following, we let $\intr(X)$ denote the interior of a set $X\subseteq\R^2$.

\begin{lemma}\label{lem:two-types}
Every effective split cut for $C$ is of one of the following types:
\begin{enumerate}
    \item a Chv\'atal cut;
    \item a cut derived from a split set $S$ such that $S \cap \intr(C_I) \neq \emptyset$; in this case, both lines delimiting $S$  intersect the same facet of $C_I$.
\end{enumerate}
\end{lemma}

\begin{proof}  Consider any effective split cut given by a split set $S$. We look at two cases: the recession cone of $C$ contains a recession direction of $S$, or not. In the second case, one of the boundary of $S$ intersects with both facets of $C$, and the other one does not intersect with $C$ since $S$ is assumed to be effective. Then the split cut is a Chv\'atal cut. In the first case, if the recession direction of $S$ is on the boundary of the recession cone of $C$, the two boundaries of $S$ are parallel to a facet of $C$. In this case as well, the split cut is a Chv\'atal cut. Finally, suppose that the interior of the recession cone of $C$ contains a recession direction of $S$. Since $C$ and $C_I$ have the same recession cone, $S$ intersects the interior of $C_I$. As no vertex of $C_I$ can be in the interior of $S$, the bounding lines of $S$ must intersect the same facet of $C_I$.
\end{proof}

\begin{definition}\label{def:ell}
Let $F_I^1,\dots, F_I^n$ be the facets of $C_I$. For every $i\in\{1,\dots,n\}$ we denote by $\ell_I^i$ the line containing $F_I^i$. Furthermore, we define $\widehat{\ell^i_I}$ as the unique line with the following properties:
\begin{enumerate}
    \item $\widehat{\ell^i_I}$ is parallel to $\ell_I^i$;
		\item $\widehat{\ell^i_I}$ contains integer points;
    \item there is no integer point strictly between $\ell_I^i$ and $\widehat{\ell^i_I}$;
    \item $\widehat{\ell^i_I}\cap C_I=\emptyset$.
\end{enumerate}
\end{definition}

Given two split cuts $H,H'$ for $C$, we say that $H$ dominates $H'$ if $C\cap H\subseteq C\cap H'$.

\begin{lemma}\label{lem:apex}
Fix $i \in \{1, \ldots, n\}$ and define the split set $S:=\conv(\ell_I^i,\widehat{\ell^i_I})$. 
If $v\in\intr(S)$, then the split cut for $C$ derived from $S$ is a Chv\'atal cut that dominates every split cut derived from a split set intersecting $F^i_I$.
\end{lemma}

\begin{proof}
Since $v\in\intr(S)$, the facets of $C$ do not intersect $\widehat{\ell^i_I}$. This implies that the split cut for $C$ derived from $S$ is a Chv\'atal cut.

Let $S'$ be any split set intersecting $F^i_I$, and denote by $h_1,h_2$ the two lines delimiting $S'$. Since there is no integer point in $\intr(S')$, both $h_1$ and $h_2$ intersect $F^i_I$. We denote by $x^1$ (resp., $x^2$) the intersection point of $h_1$ (resp., $h_2$) and $F^i_I$. Also we define $y^1$ (resp., $y^2$) as the intersection point of $h_1$ (resp., $h_2$) and $\widehat{\ell^i_I}$. Since $x^1,x^2\in C$ and $y^1,y^2\notin C$ (as $\widehat{\ell^i_I}$ does not intersect $C$), $h_1$ and $h_2$ intersect the facets of $C$ in two points contained in $\intr(S)$. This implies that any split cut derived from $S'$ is dominated by the Chv\'atal cut derived from $S$.
\end{proof}

\begin{definition}\label{def:unit}(Unit interval)
Given a line $\ell$ containing integer points, we call each closed segment whose endpoints are two consecutive integer points of $\ell$ a {\em unit interval of $\ell$}.
\end{definition}

\begin{observation}\label{obs:intersect-interval}
Any split set can intersect at most one unit interval of a given line not parallel to the split.
\end{observation}

\begin{definition}
Fix $i \in \{1, \ldots, n\}$. Given a unit interval $J$ of $\ell_I^i$ and a unit interval $\hat{J}$ of $\widehat{\ell^i_I}$, there exists a unique parallelogram of area 1 having $J$ and $\hat{J}$ as two of its sides. We denote by $S(J,\hat{J})$ the split set delimited by the lines containing the other two sides of this parallelogram. If $S(J,\hat{J})$ is effective, we denote by $H(J,\hat{J})$ the split cut for $C$ derived from $S(J,\hat{J})$.
\end{definition}

\begin{lemma}\label{lem: effective-interval}
Fix $i \in \{1, \ldots, n\}$ such that $v\notin\intr(\conv(\ell_I^i,\widehat{\ell^i_I}))$. Then there exists a unique unit interval $\hat{J}$ of $\widehat{\ell_I^i}$ such that  $\hat{J} \cap C\neq \emptyset$. Furthermore, for each unit interval $J$ of $\ell_I^i$ contained in $F_I^i$, $S(J,\hat{J})$ is an effective split set. 
\end{lemma}

\begin{proof} The existence of $\hat J$ follows from the assumption $v\notin\intr(\conv(\ell_I^i,\widehat{\ell^i_I}))$.
Furthermore, $\hat J$ is unique because, by definition of $\widehat{\ell^i_I}$, there are no integer points in $C\cap \widehat{\ell^i_I}$. 

We now prove that for each unit interval $J$ of $\ell_I^i$ contained in $F_I^i$, $S(J,\hat{J})$ is an effective split cut. Up to a unimodular transformation, we can assume that $J=\{x\in\R^2:x_2=0,\,0\le x_1\le1\}$ and $\hat{J}=\{x\in\R^2:x_2=1,\,0\le x_1\le1\}$. Then the split set $S(J,\hat{J})$ is defined by the inequalities $0\le x_1\le1$.

Since the second coordinate of $v$ is $v_2\ge1$ and $C_I$ is contained in the half-plane defined by $x_2\le0$ (as this inequality induces facet $F_I^i$ of $C_I$), it follows that both facets of $C$ intersect the lines defined by $x_2=0$ and $x_2=1$.
Thus one facet of $C$ contains points $(a_1,0)$ and $(b_1,1)$, and the other facet contains points $(a_2,0)$ and $(b_2,1)$, where $a_1\leq 0$, $a_2\geq 1$ and $0<b_1<b_2<1$. It is now straightforward to verify that $v$, which is the intersection point of the two facets, satisfies $0<v_1<1$. This shows that $S(J,\hat{J})$ is an effective split set.
%
%
%
%
\end{proof}

\begin{figure}
\includegraphics[width=0.65\textwidth]{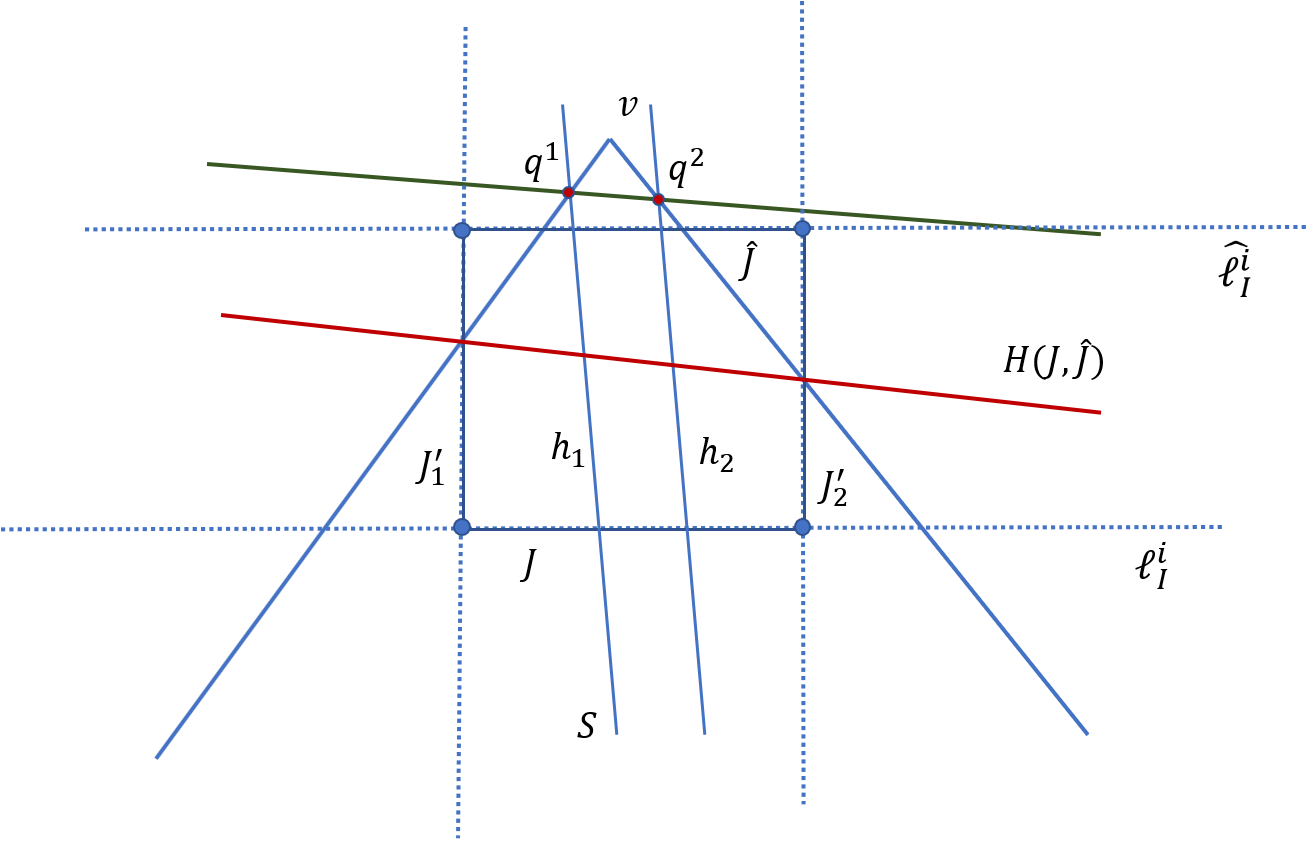}\centering
\caption{Illustration of the notation used in the proof of \cref{lem:find-canonical}. The red cutting plane $H(J,\hat J)$ dominates the dark green cutting plane derived from the split $S$.}
\label{fig:lem312}
\end{figure}

\begin{lemma}\label{lem:find-canonical}
Fix $i \in \{1, \ldots, n\}$. Let $S$ be a split set that gives an effective split cut of type 2 from \cref{lem:two-types}, where $S$ intersects  $F^i_I$. Suppose that this split cut is not dominated by a Chv\'atal cut. Let $J$ be the unit interval of $\ell_I^i$ that intersects both lines delimiting $S$ (see \cref{obs:intersect-interval}), and let $\hat{J}$ be the unit interval of $\widehat{\ell_I^i}$ such that $\hat{J} \cap C\neq \emptyset$ (see \cref{lem: effective-interval}). Then:
\begin{enumerate}[(i)]
\item  both lines delimiting $S$ intersect $\hat{J}$;
\item any cut produced by $S$ is dominated by $H(J,\hat{J})$.
\end{enumerate}
\end{lemma}

\begin{proof}
Up to a unimodular transformation, we can assume that $J=\{x\in\R^2:x_2=0,\,0\le x_1\le1\}$ and $\hat{J}=\{x\in\R^2:x_2=1,\,0\le x_1\le1\}$.
Since the split cut derived from $S$ is not dominated by a Chv\'atal cut, by \cref{lem:apex} the apex $v$ does not lie strictly between $\ell_I^i$ and $\widehat{\ell^i_I}$. In other words, $v_2\ge1$. Furthermore $0<v_1<1$, as shown in the proof of \cref{lem: effective-interval}.

Let $h_1$ and $h_2$ be the lines delimiting $S$, and define the segments $J'_1:=\{x\in\R^2: x_1=0,\, 0\leq x_2\leq 1\}$ and $J'_2:=\{x\in\R^2: x_1=1, 0\leq x_2\leq 1\}$. Since both $h_1$ and $h_2$ intersect $J$ and there is no integer point strictly between $h_1$ and $h_2$, we have that $h_1\cup h_2$ can contain points from the relative interior of at most one of $J'_1$, $J'_2$ and $\hat{J}$.

Assume that $h_1$ and $h_2$ intersect the relative interior of $J'_1$. Since $h_1$ and $h_2$ also intersect $J$, we have
\[h_1=\{x\in\R^2:x_2=ux_1+r_1\}, \qquad h_2=\{x\in\R^2:x_2=ux_1+r_2\},\]
for some $u<0$, and $r_1,r_2\in \R$.

Given any $\bar x\in h_1\cap \{x\in\R^2: x_2\ge1\}$, we have
\[\bar x_1=\frac{\bar x_2-r_1}u\le\frac{1-r_1}{u}<0.\]
Thus $h_1\cap \{x\in\R^2: x_2\ge1\}\subseteq \{x\in\R^2: x_1<0\}$. Similarly, $h_2\cap \{x\in\R^2: x_2\ge1\}\subseteq \{x\in\R^2: x_1<0\}$. As $0<v_1<1$ and $v_2\ge 1$, it follows that $v$ does not lie strictly between $h_1$ and $h_2$, a contradiction.

A similar argument shows that $h_1$ and $h_2$ do not intersect the relative interior of $J'_2$. It follows that $h_1$ and $h_2$ intersect $\hat{J}$, and (i) is proven.\smallskip

We now prove (ii).
Since, by part (i), each of $h_1$ and $h_2$ intersects both $J$ and $\hat J$, each of $h_1$ and $h_2$ intersects the boundary of $C$. Moreover, because $S$ is an effective split set, $h_1\cup h_2$ intersects the boundary of $C$ in at most two points. It follows that each of $h_1$ and $h_2$ intersects the boundary of $C$ in a single point, say $q^1$ and $q^2$, respectively. Note that $q^1_2>0$ and $q^2_2>0$, because $h_1$ and $h_2$ intersect $J$. Label $q^1$ and $q^2$ in such a way that $q^1$ (resp., $q^2$) belongs to the facet of $C$ contained in the half-plane $x_1\le v_1$ (resp., $x_1\ge v_1$). See Figure~\ref{fig:lem312}.

If $0<q^j_2<1$ for some $j\in\{1,2\}$, then $0<q^j_1<1$, because $h_1$ and $h_2$ intersect both $J$ and $\hat J$. If $q^j_2\geq 1$, then again $0<q^j_1<1$, as $\{x\in C:x_2\geq 1\}\subseteq \{x\in\R^2:0<x_1<1\}$.

The split set $S(J,\hat{J})$ is effective by \cref{lem: effective-interval}, and its boundary intersects the facets of $C$ in two points $r^1,r^2$ that satisfy $r^1_1=0$ and $r^2_1=1$. Then $r^1$ (resp., $r^2$) is further from the apex than $q^1$ (resp., $q^2$) is, as $q^1$, $q^2$ and $v$ all satisfy $0<x_1<1$. It follows that the cut $H(J,\hat{J})$ dominates any split cut derived from $S$. 
\end{proof}

\begin{figure}
\includegraphics[width=0.75\textwidth]{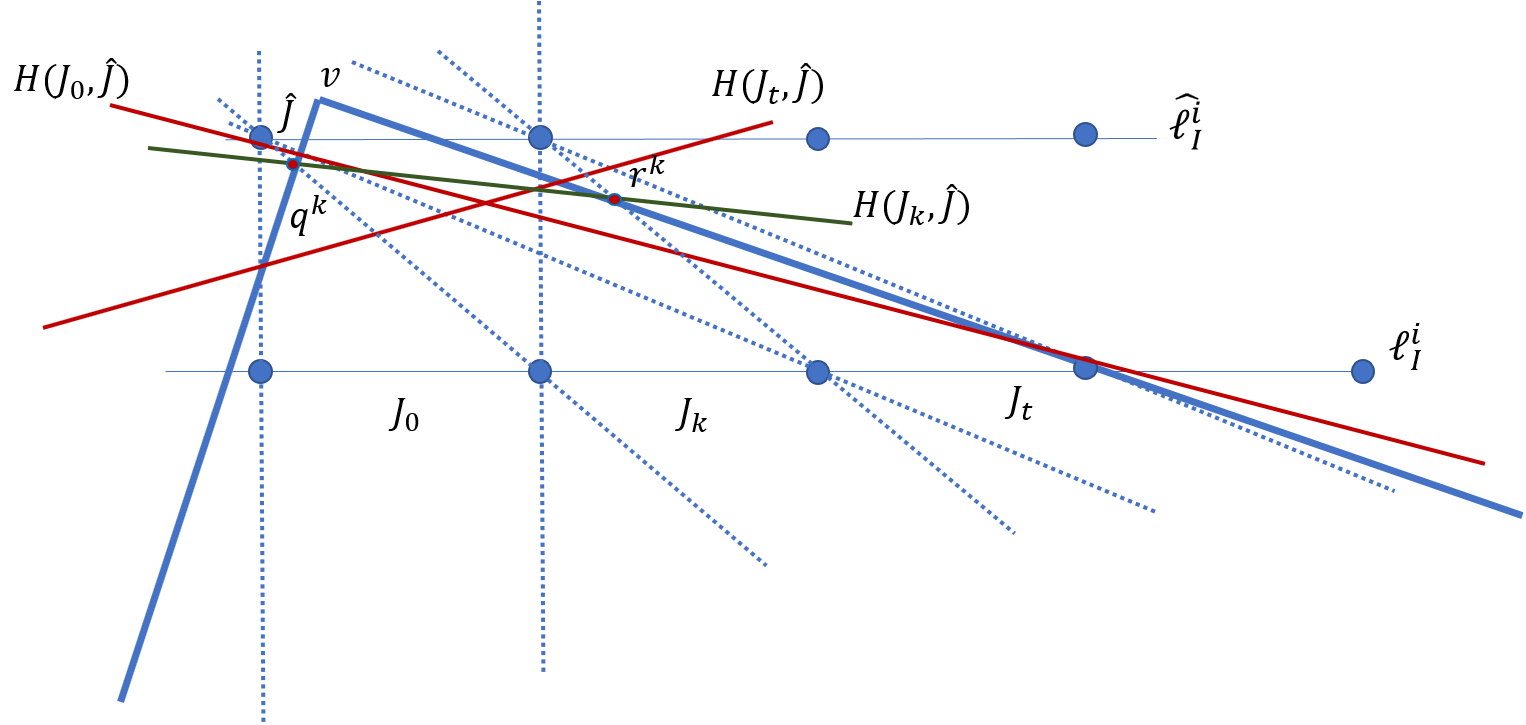}\centering
\caption{Illustration of the notation used in the proof of \cref{lem:two-sides-dominate}. The intersection of the two red cutting planes $H(J_0, \hat J)$ and $H(J_t, \hat J)$ from the leftmost and rightmost splits dominates  all other split cuts $H(J_k, \hat J_0)$, $k=0, \ldots, t$, illustrated here in dark green. }
\label{fig:lem313}
\end{figure}

\begin{lemma}\label{lem:two-sides-dominate}
Fix $i \in \{1, \ldots, n\}$ such that $v\notin\intr(\conv(\ell_I^i,\widehat{\ell^i_I}))$, and let $\hat J$ be as in \cref{lem: effective-interval}. 
Write $F^i_I=J_0\cup\dots\cup J_t$, where $J_0,\dots,J_t$ are the unit intervals contained in $F^i_I$ ordered consecutively  (i.e, $J_{k-1}\cap J_k$ only contains a single point for every $k\in\{1,\dots,t\}$). Then every point that violates $H(J_k,\hat{J})$ for some $k\in\{0,\dots,t\}$ also violates $H(J_0,\hat{J})$ or $H(J_{t},\hat{J})$.
\end{lemma}

See Figure~\ref{fig:lem313} for an illustration of Lemma~\ref{lem:two-sides-dominate} and its proof.

\begin{proof}
Up to a unimodular transformation, we can assume that $\hat{J}=\{x\in\R^2:x_2=1,\,0\le x_1\le1\}$ and $J_k=\{x\in\R^2:x_2=0,\,k\le x_1\le k+1\}$ for every $k\in\{0,\dots,t\}$.
As argued in the proof of \cref{lem: effective-interval}, one facet $G_1$ of $C$ contains points $(a_1,0)$ and $(b_1,1)$, and the other facet $G_2$ contains points $(a_2,0)$ and $(b_2,1)$, where $a_1\leq 0$, $a_2\geq 1$ (in fact, $a_2 - 1 \geq t \geq 0$) and $0<b_1<b_2<1$.
Then the lines containing $G_1$ and $G_2$ are defined by the equations $x_1+(a_1-b_1)x_2=a_1$ and $x_1+(a_2-b_2)x_2=a_2$, respectively.

Given any $k\in\{0,\dots,t\}$, the lines delimiting the split set $S(J_k,\hat{J})$ are defined by the equations $x_1+kx_2=k$ and $x_1+kx_2=k+1$. The intersection points of the former line with $G_1$ and of the latter line with $G_2$ are respectively the following:
\[q^k=\left(\frac{b_1 k}{b_1 - a_1 + k},\frac{k-a_1}{b_1 - a_1 + k}\right), \qquad
r^k=\left(\frac{b_2 k+b_2-a_2}{b_2 - a_2 + k},\frac{k+1-a_2}{b_2 - a_2 + k}\right).\]

Consider any $k\in\{0,\dots,t\}$. Since $H(J_k,\hat{J})$ is a half-plane that does not contain $v$, it is defined by an inequality of the form $c^k(x-v)\ge1$, where $c^k\in\R^2$. Note that and $c^k(q^k-v)=c^k(r^k-v)=1$, as $q^k$ and $r^k$ belong to the line delimiting $H(J_k,\hat{J})$.

Let $\bar x$ be any point in $G_1$ and $k\in\{0,\dots,t\}$. Then $\bar x=v+\mu(q^k-v)$ for some $\mu\ge0$ and therefore
\[c^k(\bar x-v)=\mu c^k(q^k-v)=\mu=\frac{v_1-\bar x_1}{v_1-q^k_1}.\]

We claim that, for fixed $\bar x\in G_1$, the above right hand side is a concave function of $k$ when $k$ is considered as a continuous parameter in $[0,t]$. To simplify the argument, we will show that $\frac{v_1-b_1}{v_1-q^k_1}$ is a concave function of $k$: this is sufficient to establish the claim, as $v_1-b_1>0$ and $v_1-\bar x_1>0$.

We can calculate
\[v_1-q^k_1=v_1-\frac{b_1 k}{b_1 - a_1 + k} = \frac{(v_1-b_1)k+v_1b_1-v_1a_1}{b_1 - a_1 + k},\]
from which we obtain
\begin{equation*}
\begin{split}
\frac{v_1-b_1}{v_1-q^k_1}&=\frac{(v_1-b_1)k+(v_1-b_1)(b_1 - a_1)}{(v_1-b_1)k+v_1b_1-v_1a_1}\\
&=1+\frac{(v_1-b_1)(b_1 - a_1)-(v_1b_1-v_1a_1)}{(v_1-b_1)k+v_1b_1-v_1a_1}\\
&=1+\frac{b_1(a_1-b_1)}{(v_1-b_1)k+v_1b_1-v_1a_1}.
\end{split}
\end{equation*}
Since $b_1>0$, $a_1-b_1<0$ and $v_1-b_1>0$, the last fraction above is of the form $\frac{\alpha}{\beta k+\gamma}$, where $\alpha<0$, $\beta>0$ and $\gamma\in\R$. It is immediate to verify that such a function of $k$ is concave for $k>-\frac\gamma\beta$. In our context, this condition reads $k>-\frac{v_1(b_1-a_1)}{v_1-b_1}$, which is a negative number. Thus the function $k\mapsto c^k(\bar x-v)$ is concave over the domain $[0,t]$.
A similar argument shows that, for any fixed $\bar x\in G_2$, the function $k\mapsto c^k(\bar x-v)$ is concave over $[0,t]$  (using the fact that $t \leq a_2 - 1$).

If $\bar x$ is any point in $C$, then we can write $\bar x=\lambda x^1+(1-\lambda)x^2$, where $x^1\in G_1$, $x^2\in G_2$ and $0\le\lambda\le1$. Then $c^k(\bar x-v)=\lambda c^k(x^1-v)+(1-\lambda)c^k(x^2-v)$ is a concave function of $k$. This implies that the minimum of the set $\{c^k(\bar x-v):k\in\{0,\dots,t\}\}$ is achieved for $k=0$ or $k=t$. In particular, if $\bar x$ violates $H(J_k,\hat{J})$ for some $k\in\{0,\dots,t\}$, then $\min\{c^0(\bar x-v),c^t(\bar x-v)\}\le c^k(\bar x-v)<1$, and thus $\bar x$ also violates $H(J_0,\hat{J})$ or $H(J_t,\hat{J})$.
\end{proof}

\begin{theorem}\label{thm:IH-and-CG}
The number of facets of the split closure of a translated cone $C\subseteq\R^2$ is at most twice the number of facets of $C_I$ plus the number of facets of the Chv\'atal closure of $C$.
\end{theorem}

\begin{proof} Let $E \subseteq \{1, \ldots, n\}$ be the index set of the facets of $C_I$ such that $v$ does not lie strictly between $\ell_I^i$ and $\widehat{\ell^i_I}$. By \cref{lem: effective-interval}, for every $i\in E$ there exists a unit interval $\hat J^i$ of $\widehat{\ell^i_I}$ such that $\hat J^i \cap C \neq \emptyset$. Moreover, let $J^i_0, \ldots J^i_{t_i}$ be the unit intervals of $\ell^i_I$ contained in $F^i_I$, ordered consecutively. Let $Q$ denote the Chv\'atal closure of $C$. We show that the split closure of $C$ is given by
\begin{equation}\label{eq:split-closure}
Q \cap \bigcap_{i \in E} \left(H(J^i_0, \hat J^i) \cap H(J^i_{t_i}, \hat J^i)\right),
\end{equation}
which suffices to prove the theorem.

Consider any split cut $H$ derived from a split set $S$. Since \cref{eq:split-closure} is contained in $Q$, we may assume that $H$ is not dominated by a Chv\'atal cut. Therefore it must be of type 2 in \cref{lem:two-types}, and thus there is a facet $F^i_I$ of $C_I$ that intersects the two lines delimiting $S$. By \cref{lem:apex}, $i \in E$. By \cref{lem:find-canonical} part (ii), $H$ is dominated by $H(J^i_k, \hat J^i)$ for some $k\in\{0,\dots,t_i\}$. By \cref{lem:two-sides-dominate}, $H(J^i_k, \hat J^i)$ is in turn dominated by $H(J^i_0, \hat J^i) \cap H(J^i_{t_i}, \hat J^i)$.
\end{proof}

To conclude the proof of \cref{thm:split-closure}, we note that by \cref{lem:corner}, \cref{thm:IH-and-CG} and \cref{thm:2D-complexity}, the number of inequalities needed to define the split closure of $P$ is polynomial in $m$, $\log\|A\|_\infty$ and $\log\|b\|_\infty$. Furthermore, the above arguments show that the size of every inequality is polynomially bounded. (However, it is known that also in variable dimension every facet of the split closure of a polyhedron $P$ is polynomially bounded; see, e.g., \cite[Theorem 5.5]{conforti2014integer}.)

\begin{remark}\label{rem:int-hull}
Given a translated cone $C\subseteq\R^2$, the arguments used in this section show that, for every facet $F^i_I$ of $C_I$, the split closure $C'$ of $C$ is strictly contained in the interior of the half-plane delimited by $\widehat\ell^i_I$ and containing $C$ (where we adopt the notation introduced in \cref{def:ell}). This implies that the Chv\'atal closure of $C'$ is $C_I$. In particular, the split rank of $C$ is at most 2.

Now let $P$ be a polyhedron in $\R^2$. It is folklore that the integer hull of $P$ is the intersection of the integer hulls of all the corner relaxations of $P$. (This is not true in higher dimensions.) Then, by the previous argument, the split rank of $P$ is at most 2.
\end{remark}

\section*{Acknowledgments}
We are very grateful for insightful comments on content and presentation from two anonymous reviewers. These helped a lot to improve the paper.

\bibliographystyle{plain}
\bibliography{full-bib}
\end{document}